\documentclass[11pt]{article}
\usepackage[T1]{fontenc}
\usepackage{lmodern}
\usepackage[margin=1.08in]{geometry}
\usepackage{amsmath,amssymb,amsthm,mathtools}
\usepackage{microtype}
\usepackage[hidelinks]{hyperref}
\ifdefined\pdfinfoomitdate\pdfinfoomitdate=1\fi
\ifdefined\pdfsuppressptexinfo\pdfsuppressptexinfo=15\fi
\ifdefined\pdftrailerid\pdftrailerid{}\fi
\hypersetup{pdftitle={Symmetric-determinant apolarity in odd characteristic},
  pdfauthor={Ying Xie},pdfsubject={Commutative algebra},
  pdfkeywords={apolarity, symmetric determinant, positive characteristic, Lefschetz properties},
  pdfcreator={},pdfproducer={},pdfstartview={}}
\newtheorem{theorem}{Theorem}[section]
\newtheorem{proposition}[theorem]{Proposition}
\newtheorem{lemma}[theorem]{Lemma}
\newtheorem{corollary}[theorem]{Corollary}
\theoremstyle{definition}

\newtheorem{example}[theorem]{Example}
\theoremstyle{remark}

\numberwithin{equation}{section}
\DeclareMathOperator{\Ann}{Ann}
\DeclareMathOperator{\Sym}{Sym}
\DeclareMathOperator{\Hom}{Hom}
\DeclareMathOperator{\End}{End}

\DeclareMathOperator{\rank}{rank}
\DeclareMathOperator{\tr}{tr}
\DeclareMathOperator{\SL}{SL}
\DeclareMathOperator{\GL}{GL}
\newcommand{\kk}{K}
\newcommand{\Cat}{\mathrm{Cat}}
\newcommand{\one}{\mathbf{1}}
\newcommand{\set}[1]{\{#1\}}
\newcommand{\eps}{\varepsilon}
\title{Symmetric-determinant apolarity\protect\\ in odd characteristic}
\author{Ying Xie\\[3pt]
  \small Kennesaw State University\\
  \small\texttt{yxie2@kennesaw.edu}}
\date{}

\begin{document}
\maketitle

\begin{abstract}
Let $D_N$ be the determinant of the generic symmetric $N\times N$ matrix,
and let differential operators act by ordinary differentiation over a field
of odd characteristic $p$. We determine the apolar ideal of $D_N$ in every
size. In addition to the classical quadratic relations, it is generated by
one $(p-1)\times(p-1)$ determinant of differential variables for each set of
$2p-2$ indices. These additional generators are minimal modulo the
quadratic ideal. An exterior-algebra realization identifies the remaining
apolar relations with radicals of invariant pairings for $\SL_2$.
Semisimplification gives a Hilbert series in terms of Dyck paths of height
at most $p-2$, and the Steinberg symmetrizer gives the explicit generators
of the entire radical. We also compute the ranks of powers of a general
linear form. The weak Lefschetz property holds exactly when $N\leq 2p-2$,
and the strong Lefschetz property holds exactly when $N<p$.
\end{abstract}

\section{Introduction and main results}

Let $\kk$ be a field and let $X=(x_{ij})_{1\leq i,j\leq N}$ be the
generic symmetric matrix, where $x_{ji}=x_{ij}$ and the variables
$x_{ij}$ with $i\leq j$ are algebraically independent. Set
\[
 R=\kk[x_{ij}:1\leq i\leq j\leq N],\qquad
 S=\kk[d_{ij}:1\leq i\leq j\leq N],\qquad D_N=\det X.
\]
Throughout, $d_{ij}$ acts on $R$ as the ordinary partial derivative
$\partial/\partial x_{ij}$, and $d_{ji}$ means $d_{ij}$. The apolar algebra is
\[
 A_N=S/\Ann(D_N),\qquad
 \Ann(D_N)=\set{f\in S:f(\partial)D_N=0}.
\]
The grading on $S$ assigns degree one to every $d_{ij}$. We assume $N\geq1$.

In characteristic zero, Shafiei determined the apolar ideal of the
symmetric determinant and proved that it is generated by explicit
quadrics \cite[Theorem~3.10]{Shafiei}. Ordinary differentiation in positive
characteristic introduces a further issue: the top multiplication pairing
on the quadratic quotient can become degenerate. The first degeneracy in
characteristic $p>2$ occurs at $N=2p-2$. We describe this degeneracy in all
degrees and show that its first occurrence generates all subsequent
relations.

Write $Q_N\subset S$ for the ideal generated by
\begin{align}
 &d_{ii}^{\,2},\quad d_{ii}d_{ij}\quad(i\ne j),\notag\\
 &d_{ij}^{\,2}+2d_{ii}d_{jj}\quad(i\ne j),\notag\\
 &d_{ij}d_{ik}+d_{ii}d_{jk}\quad(i,j,k\text{ distinct}),\notag\\
 &d_{ij}d_{kl}+d_{ik}d_{jl}+d_{il}d_{jk}
       \quad(i,j,k,l\text{ distinct}).\label{eq:quadrics}
\end{align}
Suppose $\operatorname{char}\kk=p>2$ and put $m=p-1$. For a subset
$U=\{u_1<\cdots<u_{2m}\}\subseteq[N]$, define
\begin{equation}\label{eq:generator}
 F_U=\det\bigl(d_{u_i,u_{m+j}}\bigr)_{1\leq i,j\leq m}.
\end{equation}
Thus the row and column index sets are disjoint and together form $U$.

\begin{theorem}[Apolar ideal]\label{thm:ideal}
Over a field of odd characteristic $p$,
\begin{equation}\label{eq:ideal}
 \Ann(D_N)=Q_N+(F_U:U\subseteq[N],\ |U|=2p-2).
\end{equation}
The classes of the displayed $F_U$ form a basis of
$(\Ann(D_N)/Q_N)_{p-1}$ and minimally generate $\Ann(D_N)/Q_N$ as an
$S$-module. In particular, there are exactly $\binom{N}{2p-2}$ additional
generators modulo $Q_N$. If $p\geq5$, the apolar ideal is generated by
quadrics if and only if $N<2p-2$; its only possible minimal generator
degrees are $2$ and $p-1$. In characteristic $3$, the entire apolar ideal
is generated by quadrics. In characteristic zero, $\Ann(D_N)=Q_N$.
\end{theorem}

A Dyck path of length $2s$ is a walk from $0$ to $0$ with steps $+1$ and
$-1$ which never goes below $0$. Let $b_s(p)$ denote the number of these
paths whose height is at most $p-2$, and put $b_0(p)=1$. We use
$\binom{a}{b}=0$ when $b<0$ or $b>a$, for $a\geq0$.

\begin{theorem}[Hilbert series]\label{thm:hilbert}
Over a field of odd characteristic $p$, the Hilbert series of $A_N$ is
\begin{equation}\label{eq:hilbert}
 H_{A_N}(t)=\sum_{s=0}^{\lfloor N/2\rfloor}
       \binom{N}{2s}b_s(p)t^s(1+t)^{N-2s}.
\end{equation}
If $P$ is the adjacency matrix of the path with vertices $0,1,\ldots,p-2$,
then
\begin{equation}\label{eq:length}
 \dim_\kk A_N=((2I+P)^N)_{0,0}
 =\sum_{j\in\mathbb Z}
  \left[\binom{2N}{N+2pj}-\binom{2N}{N+2pj+2}\right].
\end{equation}
The sum is finite under the binomial convention above. In characteristic
zero, the same Hilbert-series formula holds with
$b_s(0)=\Cat_s=\frac{1}{s+1}\binom{2s}{s}$; in that case
\begin{equation}\label{eq:charzero}
 \dim_\kk(A_N)_k=\frac{1}{N+1}\binom{N+1}{k}\binom{N+1}{k+1},
 \qquad \dim_\kk A_N=\Cat_{N+1}.
\end{equation}
\end{theorem}

For a standard graded Artinian algebra $A$, a linear form $L\in A_1$ is
weak Lefschetz if multiplication $A_k\xrightarrow{L}A_{k+1}$ has maximal
rank for every $k$. It is strong Lefschetz if
$A_k\xrightarrow{L^q}A_{k+q}$ has maximal rank for every $k,q\geq0$.
The corresponding properties mean that such a form exists over the
ground field.

\begin{theorem}[Lefschetz properties]\label{thm:lefschetz}
Over any field of odd characteristic $p$, the algebra $A_N$ has the weak
Lefschetz property exactly when $N\leq2p-2$, and it has the strong
Lefschetz property exactly when $N<p$. Whenever either property holds,
it is realized by
\[
 \ell=d_{11}+\cdots+d_{NN}.
\]
At $N=2p-1$, the map
$\ell:(A_N)_{p-1}\longrightarrow(A_N)_p$ has a one-dimensional kernel.
In characteristic zero, $\ell$ is strong Lefschetz for every $N$.
\end{theorem}

The proof separates a characteristic-independent quadratic algebra from
the characteristic-dependent radical of its top pairing. The former is a
subalgebra of an exterior algebra. Its multigraded pieces are spaces of
invariant tensors for the standard representation of $\SL_2$. The rank of
the invariant pairing is determined by the semisimplification of tilting
modules. To obtain ideal generators, we use the Steinberg symmetrizer and
then translate its tensor-ideal property into multiplication in $S/Q_N$.
Finally, multiplication by $\ell$ decomposes into subset-inclusion maps,
whose ranks are given by Wilson's formula in the form proved by Frankl
\cite[Theorem~4.1]{Frankl}.

The ordinary differential convention is essential. The characteristic
assumption stated in \cite{Shafiei} includes all $p>2$, whereas the
quadratic presentation requires the additional relations in
\eqref{eq:ideal} in small characteristic. We use its characteristic-zero
presentation as the point of comparison. Section~\ref{sec:examples}
gives the first explicit discrepancy. No divided-power action is used
anywhere in this paper.

\section{The exterior model and its quadratic presentation}
\label{sec:exterior}

Assume for this section that $2$ is invertible in $\kk$. Let
\[
 \Lambda=\bigwedge\left(\bigoplus_{i=1}^N
                 (\kk a_i\oplus\kk b_i)\right),\qquad
 u_i=a_i b_i,\qquad z_{ij}=a_i b_j+a_j b_i\quad(i\ne j).
\]
All products in $\Lambda$ are exterior products; we omit wedge symbols.
The elements $u_i$ and $z_{ij}$ have even exterior degree and commute.
There is therefore an algebra homomorphism
\begin{equation}\label{eq:phi}
 \phi:S\longrightarrow\Lambda,\qquad
 d_{ii}\longmapsto u_i,\quad d_{ij}\longmapsto z_{ij} (i<j).
\end{equation}
Give $a_i,b_i$ vertex degree $e_i\in\mathbb Z^N$. On $S$ this means
$\deg_v d_{ii}=2e_i$ and $\deg_v d_{ij}=e_i+e_j$. All the relations
in \eqref{eq:quadrics} are homogeneous for this grading.

For a finite ordered set $E$ of even size $2s$, let $W(E)$ be the span of
the products
\[
 z_M=\prod_{\{i,j\}\in M}z_{ij},
\]
where $M$ runs through the perfect matchings of $E$. The space $W(E)$
lies in the exterior component of vertex degree one at each vertex of
$E$ and zero elsewhere. Write $u_T=\prod_{i\in T}u_i$.

\begin{proposition}[Quadratic normal form]\label{prop:normal}
The kernel of $\phi$ is $Q_N$. If $C_N=S/Q_N$, then its vertex-multigraded
pieces are
\begin{equation}\label{eq:pieces}
 (C_N)_\alpha=
 \begin{cases}
  u_TW(E),&\alpha=\one_E+2\one_T,
       \ E\cap T=\varnothing,\ |E|=2s,\\
  0,&\text{otherwise}.
 \end{cases}
\end{equation}
The products associated with noncrossing matchings of $E$ form a basis
of $W(E)$. Consequently $\dim_\kk W(E)=\Cat_s$, independently of the
odd characteristic.
\end{proposition}

\begin{proof}
Exterior multiplication gives the relations in \eqref{eq:quadrics}.
For example,
\[
 z_{ij}^{\,2}=-2u_i u_j,\qquad
 z_{ij}z_{jk}=-u_jz_{ik},\qquad
 z_{ij}z_{kl}+z_{ik}z_{jl}+z_{il}z_{jk}=0.
\]
This proves $Q_N\subseteq\ker\phi$.

Represent a monomial in $S$ by a multigraph on $[N]$, with a loop at $i$
for each factor $d_{ii}$ and an edge $ij$ for each factor $d_{ij}$.
A loop contributes two to the vertex degree. A monomial with vertex
degree at least three at some vertex vanishes modulo $Q_N$. Indeed, a
loop together with any further incident factor is zero. A repeated
edge can be replaced by $-2d_{ii}d_{jj}$, reducing to that case. For
three distinct edges incident at $i$, replacing the first two by
$-d_{ii}d_{jk}$ again gives zero.

If every vertex degree is at most two, each nontrivial component is a
path, a cycle, or a loop. Repeated use of
$d_{ij}d_{jk}=-d_{jj}d_{ik}$ reduces a path to its endpoint edge times
the diagonal variables at its internal vertices, with a sign. A cycle
reduces to a signed factor $2$ times the product of the diagonal
variables at its vertices. It follows that the component of vertex
degree $\one_E+2\one_T$ is spanned by
$\bigl(\prod_{i\in T}d_{ii}\bigr)
 \bigl(\prod_{\{i,j\}\in M}d_{ij}\bigr)$,
with $M$ a matching on $E$.

The four-vertex relation straightens these matching products into
noncrossing ones. This is the usual bracket straightening: place the
ordered vertices on a line and resolve each crossing by the other two
pairings of its four ends. Drawing the resolved arcs before straightening
them removes a crossing, so iteration terminates. Equivalently, this is
the multilinear Pl\"ucker straightening for two-dimensional brackets.

For independence, order the exterior generators by vertex. In the
component with one generator at each vertex, identify a monomial with
its word in $a,b$. With $a>b$ in lexicographic order, the leading word
of $z_M$ has $a$ at each left endpoint and $b$ at each right endpoint,
with coefficient $1$ or $-1$. Noncrossing matchings have distinct leading
words: these are exactly the Dyck words, and the matching is recovered
by the usual stack pairing. Their images are therefore linearly
independent in every characteristic. Multiplication by $u_T$ on disjoint
vertices preserves independence. The spanning set just obtained in
$S/Q_N$ consequently maps to a basis of the image, proving the claim.
\end{proof}

Let $V=\kk a\oplus\kk b$ with the symplectic form
$\eps(a,b)=1$. Its symplectic group is $\SL_2$. Ordered exterior
multiplication identifies the component with one letter at each vertex
of $E$ with $V^{\otimes 2s}$. Under this identification $z_{ij}$ is the
alternating invariant tensor on positions $i,j$, with the signs supplied
by the order of the exterior factors. The characteristic-free first
fundamental theorem for symplectic invariants gives
\begin{equation}\label{eq:invariants}
 W(E)=(V^{\otimes 2s})^{\SL_2}.
\end{equation}
One may apply the multilinear part of
\cite[Theorem~17]{Hashimoto}; in dimension two its generators are the
brackets used above. We shall also use its consequence that equivariant
maps between tensor powers of $V$ are spanned by diagrams made from
symplectic evaluations, coevaluations, and permutations of factors.
This follows by identifying $V$ with $V^*$ and bending input factors
into output factors.

\section{The apolar pairing and its modular rank}
\label{sec:pairing}

Let $\tau:\Lambda\to\kk$ extract the coefficient of
$u_{[N]}=u_1\cdots u_N$. Consider the finite expression
\begin{equation}\label{eq:exponential}
 \mathcal E(X)=\prod_{i=1}^N(1+x_{ii}u_i)
       \prod_{1\leq i<j\leq N}
           \left(1+x_{ij}z_{ij}+\frac{x_{ij}^2z_{ij}^{\,2}}2\right).
\end{equation}
It is an element of $R\otimes_\kk\Lambda$. The quadratic denominators
are units under our characteristic assumption.

\begin{lemma}\label{lem:top}
The identities
\begin{equation}\label{eq:top}
 \tau(\mathcal E(X))=D_N,\qquad
 f(\partial)D_N=\tau\bigl(\phi(f)\mathcal E(X)\bigr)
       \quad(f\in S)
\end{equation}
hold over every field of characteristic different from two.
Consequently
\begin{equation}\label{eq:radical}
 \Ann(D_N)/Q_N
 =\set{c\in C_N:\tau(cc')=0\text{ for all }c'\in C_N}.
\end{equation}
\end{lemma}

\begin{proof}
The determinant identity can be checked over $\mathbb Q$ by writing
\[
 \mathcal E(X)=\exp\left(\sum_i x_{ii}a_i b_i
                   +\sum_{i<j}x_{ij}(a_i b_j+a_j b_i)\right).
\]
The coefficient of $u_{[N]}$ is the usual alternating permutation
expansion of $\det X$. The finite product \eqref{eq:exponential} and this
coefficient identity are defined over $\mathbb Z[1/2]$, so the identity
holds after specialization to every field under consideration. In
particular, no exponential series with noninvertible factorials is being
used in positive characteristic.

Direct differentiation of each factor in \eqref{eq:exponential} gives
$\partial_{ij}\mathcal E=\phi(d_{ij})\mathcal E$, since $u_i^2=0$ and
$z_{ij}^3=0$. Iteration proves the second identity in \eqref{eq:top}.

For the last assertion, let
\[
 B=\set{\beta:\ \beta_{ii}\in\{0,1\},\
                   \beta_{ij}\in\{0,1,2\}\text{ for }i<j},
 \qquad \beta!=\prod_{i\leq j}\beta_{ij}!.
\]
Expanding the finite product gives
\[
 \mathcal E(X)=\sum_{\beta\in B}
                  \frac{x^\beta}{\beta!}\phi(d^\beta),
 \qquad
 f(\partial)D_N=\sum_{\beta\in B}
     \frac{x^\beta}{\beta!}\tau\bigl(\phi(f)\phi(d^\beta)\bigr).
\]
Every $\beta!$ in these sums is a power of $2$ and hence is a unit.
Since the distinct monomials $x^\beta$ are linearly independent in
the polynomial ring $R$, the second sum vanishes if and only if
$\tau(\phi(f)\phi(d^\beta))=0$ for every $\beta\in B$. This is a
coefficientwise statement about formal polynomials, valid also over a
finite field. The elements $\phi(d^\beta)$ with $\beta\in B$ span $C_N$:
every operator monomial outside these exponent bounds has zero image,
by $u_i^2=0$ and $z_{ij}^3=0$. Thus $f(\partial)D_N=0$ if and only if
$\phi(f)$ pairs to zero with every element of $C_N$. Together with
$\ker\phi=Q_N$, this proves \eqref{eq:radical}.
\end{proof}

The radical in \eqref{eq:radical} is an ideal, by associativity of
multiplication. Only complementary vertex degrees can pair under $\tau$.
If $\alpha=\one_E+2\one_T$, its complement in $2\one_{[N]}$ is
$\one_E+2\one_{[N]\setminus(E\cup T)}$. To compute this pairing, order
the letters at the vertices of $E$ and move the two letters at each
vertex next to one another. This reordering contributes a fixed sign,
independent of the letters; the coefficient of $a_i b_i$ at each
vertex is the corresponding value of $\eps$. The induced form on
$W(E)$ is therefore, up to an overall sign, the restriction of
$\eps^{\otimes 2s}$.
Let $\mathcal R(E)$ denote its radical and put
\[
 U(E)=W(E)/\mathcal R(E).
\]
It follows that
\begin{equation}\label{eq:apolarpieces}
 (A_N)_{\one_E+2\one_T}\simeq u_TU(E).
\end{equation}
In particular, the quotient has a nondegenerate top pairing and
one-dimensional top degree $N$, generated by $u_{[N]}$. This also proves
directly that $A_N$ is Artinian Gorenstein in the ordinary differential
convention used here.

We now compute $\dim U(E)$. The representation-theoretic argument may
be made over an algebraic closure: the ranks of the matrices defining
all the spaces and pairings above are unchanged by extending scalars.

\begin{proposition}\label{prop:rank}
If $|E|=2s$ and $\operatorname{char}\kk=p>2$, then
\[
 \dim_\kk U(E)=b_s(p).
\]
In characteristic zero, the invariant pairing is nondegenerate and
$\dim_\kk U(E)=\Cat_s$.
\end{proposition}

\begin{proof}
Suppose first that $\kk$ is algebraically closed of characteristic $p$.
Let $T_r$ be the indecomposable tilting $\SL_2$-module of highest weight
$r$. Recall that a morphism $f:M\to M'$ is negligible if
$\tr(gf)=0$ for every $g:M'\to M$. The negligible tilting modules are
exactly the direct sums of $T_r$ with $r\geq p-1$, and a morphism between
tilting modules is negligible precisely when it factors through such
a module \cite[Lemma~2.6 and Corollary~2.7]{BFRSW}. The resulting quotient
category is semisimple, with simple objects the images of
$T_0,\ldots,T_{p-2}$; these modules are $\Sym^rV$ in the indicated range.
For the general semisimplification framework, see also
\cite[Sections~1--2]{BEEO}.

Set $M=V^{\otimes 2s}$, a tilting module. The pairing between
$\Hom(\kk,M)$ and $\Hom(M,\kk)$ is composition. Via the symplectic
self-duality of $V$, it is the invariant pairing on $W(E)$. Its radical
is therefore exactly the space of negligible morphisms from $\kk$ to
$M$. The dimension of $U(E)$ is the multiplicity of the unit object in
the image of $M$ in the quotient category.

For $1\leq r\leq p-2$, the Clebsch--Gordan splitting gives
\begin{equation}\label{eq:cg}
 V\otimes\Sym^rV\simeq\Sym^{r+1}V\oplus\Sym^{r-1}V.
\end{equation}
The usual multiplication and polarization maps split this decomposition
after division by $r+1$, which is invertible throughout this range.
At $r=p-2$, the first summand is $T_{p-1}$ and becomes zero in the
quotient. At $r=0$, tensoring by $V$ gives only $T_1$.
Thus tensoring by $V$ in the quotient has adjacency matrix $P$ on the
vertices $0,\ldots,p-2$. The multiplicity in question is
$(P^{2s})_{0,0}$, which counts the required bounded Dyck paths.

In characteristic zero, complete reducibility makes the invariant
pairing nondegenerate. The unrestricted Clebsch--Gordan rule counts all
Dyck paths and gives $\Cat_s$. The conclusions descend from the
algebraic closure by invariance of matrix rank under field extension.
\end{proof}

\section{Explicit generators of the radical}
\label{sec:generators}

Throughout this section $\operatorname{char}\kk=p>2$ and $m=p-1$.
Fix the ordered tensor basis of $V^{\otimes 2m}$ and define
\begin{equation}\label{eq:omega}
 \Omega_m=\sum_{\substack{w\in\{a,b\}^{2m}\\
                     \#a(w)=\#b(w)=m}} w.
\end{equation}
The tensor is symmetric under every permutation of its factors.

\begin{lemma}[Steinberg tensor]\label{lem:steinberg}
The tensor $\Omega_m$ is a nonzero invariant tensor. It is a nonzero
scalar multiple of the tensor obtained by bending the input factors of
the symmetrizer
\[
 e_m=\frac{1}{m!}\sum_{\sigma\in\mathfrak S_m}\sigma
       \ \in\End_{\SL_2}(V^{\otimes m}).
\]
It is negligible as a morphism $\kk\to V^{\otimes 2m}$.
Contraction of any two of its factors by $\eps$ is zero.
\end{lemma}

\begin{proof}
Since $m!$ is invertible, $e_m$ is the projector onto
$\Sym^mV=T_m$, the Steinberg module. Bend its input factors using
$V\simeq V^*$. Up to a fixed overall sign, the coefficient of a
balanced word having $k$ letters $a$ in the first half is
\[
 \frac{(-1)^{m-k}}{\binom{m}{k}}.
\]
Indeed, among the $m!$ permutations in the symmetrizer, exactly
$k!(m-k)!$ produce a given arrangement with that distribution of
letters; the symplectic identification supplies the displayed sign.
The congruence
$\binom{p-1}{k}=(-1)^k$ in $\kk$ shows that all balanced words have the
same nonzero coefficient. Unbalanced words have coefficient zero.
The bent tensor is therefore a nonzero multiple of $\Omega_m$.

The identity of $T_m$ is negligible, so $e_m$ and its bent tensor are
negligible as well. Finally, interchanging any selected pair of factors
fixes $\Omega_m$ and reverses the sign of contraction by the alternating
form $\eps$. The contraction is its own negative and hence is zero.
\end{proof}

For $J\subseteq E$ of size $2m$, write $\Omega_J$ for the copy of
$\Omega_m$ on the ordered positions of $J$. Products with matching
tensors on $E\setminus J$ mean tensor products inserted in their specified
positions; an exterior reordering can change an overall sign only.

\begin{proposition}[The entire invariant radical]\label{prop:radgens}
For every ordered set $E$ of size $2s$,
\begin{equation}\label{eq:radspan}
 \mathcal R(E)=
 \operatorname{span}_\kk
 \set{\Omega_J z_M:
       J\subseteq E,\ |J|=2m,\ M\text{ a matching of }E\setminus J}.
\end{equation}
In particular, $\mathcal R(E)=0$ if $s<m$, and it is one-dimensional
if $s=m$.
\end{proposition}

\begin{proof}
We may again extend scalars to an algebraic closure. By
\cite[Lemma~2.6 and Corollary~2.7]{BFRSW}, a negligible morphism between
tilting modules factors through a finite direct sum of modules $T_r$
with $r\geq m$. Each such $T_r$ is a direct summand of
\begin{equation}\label{eq:steinsummand}
 T_m\otimes V^{\otimes(r-m)}.
\end{equation}
Indeed, this is a tilting module of highest weight $r$, and its weight
$r$ space is one-dimensional. Its decomposition into indecomposable
tilting modules must consequently contain $T_r$ once. Inserting the
inclusion and projection of this summand shows that every negligible
morphism between tensor powers is a sum of morphisms factoring through
modules of the form \eqref{eq:steinsummand}.

Apply this observation to an invariant in $\mathcal R(E)$, regarded as
a negligible morphism $\kk\to V^{\otimes 2s}$. Choose equivariant maps
\[
 i:T_m\longrightarrow V^{\otimes m},\qquad
 \pi:V^{\otimes m}\longrightarrow T_m,\qquad
 \pi i=1_{T_m},\quad i\pi=e_m.
\]
For a summand $f=ba$ factoring through $T_m\otimes V^{\otimes q}$,
put $h=(i\otimes1)a$ and $g=b(\pi\otimes1)$. Then
$f=g(e_m\otimes1)h$, giving the factorization
\begin{equation}\label{eq:factorization}
 \kk\xrightarrow{h}V^{\otimes(m+q)}
 \xrightarrow{e_m\otimes 1}V^{\otimes(m+q)}
 \xrightarrow{g}V^{\otimes 2s},
\end{equation}
where $h,g$ are equivariant. Identifying $V$ with $V^*$, the multilinear
part of \cite[Theorem~17]{Hashimoto} expresses both maps as linear
combinations of pairing diagrams, as in \eqref{eq:invariants}.
By Lemma~\ref{lem:steinberg}, bending the $e_m$ box replaces it by a
nonzero scalar multiple of $\Omega_m$. Absorb this scalar into the
coefficient of the diagram. All remaining operations are symplectic
evaluations, coevaluations, and permutations of tensor factors.

For completeness, the resulting contractions can be reduced as follows.
Away from the distinguished tensor $\Omega_m$ and the output positions,
every strand has two ends and no branching. Successive evaluations and
coevaluations along a strand cancel by the duality identities; using
the alternating self-duality of $V$ can introduce only an overall sign.
Thus each component is a closed loop or a path whose endpoints are
legs of $\Omega_m$ or output positions. A path joining two legs of
$\Omega_m$ contracts those factors by $\eps$, up to sign, and gives
zero by Lemma~\ref{lem:steinberg}. In every surviving term all $2m$
legs therefore reach distinct output positions, forming a subset $J$.
The other output positions are paired among themselves, giving a
matching $M$ of $E\setminus J$, while closed loops contribute scalars.
The term is consequently a scalar multiple of $\Omega_Jz_M$.
This proves the inclusion from left to right.

Conversely, every tensor on the right is obtained from the bent
Steinberg projector by tensoring with invariant pairings and permuting
factors. Negligible morphisms form a tensor ideal, so this tensor is
negligible and belongs to $\mathcal R(E)$. For $s<m$ no such subset $J$
exists; for $s=m$ the right side is the line spanned by the nonzero
$\Omega_m$. Equality of the spaces descends to $\kk$.
\end{proof}

The preceding proposition describes an invariant tensor space. To obtain
generators in the commutative polynomial ring, we identify its distinguished
tensor with an actual polynomial in the differential variables.

\begin{lemma}\label{lem:detomega}
For every $U\subseteq[N]$ of size $2m$, the exterior image $\phi(F_U)$
is a nonzero scalar multiple of $\Omega_U$.
\end{lemma}

\begin{proof}
Relabel $U$ increasingly as $1,\ldots,2m$, with halves $L$ and $R$.
Expanding the determinant and ordering all exterior letters by vertex,
the coefficient of a balanced word with $k$ letters $a$ in $L$ is
\begin{equation}\label{eq:detcoeff}
 (-1)^{m(m-1)/2+m-k}k!(m-k)!.
\end{equation}
To see this, such a word can occur only when the $k$ positions carrying
$a$ in $L$ are matched with the $k$ positions carrying $b$ in $R$.
There are $k!(m-k)!$ compatible bijections. The sign of each bijection
in the determinant cancels the corresponding reordering sign on the
$R$ positions. Moving the $L$ letters ahead of the $R$ letters contributes
$(-1)^{m(m-1)/2}$, and the reversed bracket choices contribute
$(-1)^{m-k}$. This gives \eqref{eq:detcoeff}.

In characteristic $p$, the identity
\[
 k!(m-k)!=\frac{m!}{\binom{m}{k}}=(-1)^k m!
\]
shows that \eqref{eq:detcoeff} is independent of $k$ and nonzero.
Every unbalanced word has coefficient zero. This is the required
multiple of \eqref{eq:omega}.
\end{proof}

\begin{proof}[Proof of Theorem~\ref{thm:ideal}]
The radical in vertex degree $\one_E+2\one_T$ is
$u_T\mathcal R(E)$ by \eqref{eq:radical}--\eqref{eq:apolarpieces}.
Proposition~\ref{prop:radgens} and Lemma~\ref{lem:detomega} express each
of its elements as a linear combination of
\[
 \phi(F_J)u_Tz_M,
 \qquad J\subseteq E,\quad |J|=2m,
 \quad M\text{ a matching of }E\setminus J.
\]
These are precisely products of the proposed generators with elements
of $C_N$. Conversely, each $\phi(F_J)$ lies in the radical, which is an
ideal. This proves \eqref{eq:ideal}, using
$\ker\phi=Q_N$.

In ordinary degree $k$, a nonzero radical piece requires $s\geq m$ and
$k=s+|T|$. There is therefore no radical in degrees below $m$. In degree
$m$, the only radical pieces have $s=m$ and $T=\varnothing$. Each is a
line, and the $\binom{N}{2m}$ different supports give distinct vertex
degrees. The classes of the $F_U$ form a basis in that degree and all
are necessary as generators modulo $Q_N$. If $p\geq5$, then $m>2$, so
the quadratic part of the apolar ideal is exactly $(Q_N)_2$. If $p=3$,
the additional generators also have degree two. The assertions about
minimal degrees and quadratic generation follow.

In characteristic zero Proposition~\ref{prop:rank} gives zero radical
in every vertex degree. Lemma~\ref{lem:top} then yields
$\Ann(D_N)=Q_N$.
\end{proof}

\section{Hilbert series and the first modular defects}
\label{sec:examples}

\begin{proof}[Proof of Theorem~\ref{thm:hilbert}]
For ordinary degree $k$ in \eqref{eq:apolarpieces}, choose the $2s$
endpoints of $E$, then the $k-s$ vertices of $T$ in its complement.
Proposition~\ref{prop:rank} gives
\begin{equation}\label{eq:hilbertvalues}
 \dim_\kk(A_N)_k=
 \sum_{s=0}^{\lfloor N/2\rfloor}
 \binom{N}{2s}\binom{N-2s}{k-s}b_s(p).
\end{equation}
Summing over $k$ with weight $t^k$ proves \eqref{eq:hilbert}.

At $t=1$ the sum is
$\sum_s\binom{N}{2s}2^{N-2s}(P^{2s})_{0,0}$.
Odd powers of $P$ have zero $(0,0)$ entry, so the binomial theorem gives
$((2I+P)^N)_{0,0}$. This counts length-$N$ walks in the strip
$0,\ldots,p-2$, allowing up and down steps and two colors of horizontal
steps. For unrestricted walks, the number with displacement $r$ is
\[
 [z^r](2+z+z^{-1})^N
 =[z^{N+r}](1+z)^{2N}=\binom{2N}{N+r}.
\]
Reflection at the absorbing boundaries $-1$ and $p-1$ gives the image
displacements $2pj$ and $2pj+2$, with opposite signs. Horizontal-step
colors are preserved under reflection. Summing these contributions
proves \eqref{eq:length}.

In characteristic zero, substitute $b_s(0)=\Cat_s$ into
\eqref{eq:hilbertvalues}. For $0\leq k\leq N$, each summand becomes
\[
 \frac{N!}{s!(s+1)!(k-s)!(N-k-s)!}.
\]
Factoring out $\binom{N}{k}/(k+1)$ leaves
$\binom{k+1}{s+1}\binom{N-k}{s}$. Vandermonde's identity gives
$\binom{N+1}{k}$ for the sum of these products. This is the first
formula in \eqref{eq:charzero}. Summing over $k$ gives
$\binom{2N+2}{N}/(N+1)=\Cat_{N+1}$.
\end{proof}

\begin{corollary}\label{cor:first}
The Hilbert function in characteristic $p>2$ agrees with its
characteristic-zero value for $N<2p-2$. At $N=2p-2$, its only change is
a decrease by one in degree $p-1$.
\end{corollary}

\begin{proof}
A Dyck path of semilength $s<p-1$ cannot reach height $p-1$.
For $s=p-1$, exactly one path does: all up steps followed by all down
steps. Therefore $b_s(p)=\Cat_s$ for $s<p-1$ and
$b_{p-1}(p)=\Cat_{p-1}-1$. The claim follows from
\eqref{eq:hilbert}.
\end{proof}

\begin{example}[Characteristic three]\label{ex:three}
For $p=3$, the allowed heights are $0$ and $1$, so $b_s(3)=1$ for all
$s$. Hence
\begin{equation}\label{eq:p3length}
 H_{A_N}(t)=\sum_s\binom{N}{2s}t^s(1+t)^{N-2s},\qquad
 \dim_\kk A_N=\frac{3^N+1}{2}.
\end{equation}
At $N=4$, the additional relation is
\[
 F_{\{1,2,3,4\}}=d_{13}d_{24}-d_{14}d_{23}.
\]
It does not belong to $Q_4$, since its exterior image is a nonzero
multiple of $\Omega_2$. Direct differentiation over the integers gives
\begin{equation}\label{eq:p3direct}
 (d_{13}d_{24}-d_{14}d_{23})D_4
       =6(x_{13}x_{24}-x_{14}x_{23}).
\end{equation}
This vanishes in characteristic three. The Hilbert function is
$(1,10,19,10,1)$, compared with $(1,10,20,10,1)$ in characteristic zero.
Thus the classical list of quadrics is incomplete in this characteristic,
although the full ideal is still generated in degree two.
\end{example}

\begin{example}[The first higher-degree generator]
For $p=5$ and $N=8$, the first additional generator is the quartic
\[
 \det(d_{ij})_{1\leq i\leq4,\ 5\leq j\leq8}.
\]
It is necessary modulo the quadratic ideal. The middle Hilbert value
decreases from $1764$ to $1763$, and the total dimension decreases from
$\Cat_9=4862$ to $4861$.
\end{example}

\begin{corollary}\label{cor:unimodal}
In every odd characteristic, the Hilbert function of $A_N$ is symmetric
and unimodal, with center $N/2$.
\end{corollary}

\begin{proof}
Every polynomial $t^s(1+t)^{N-2s}$ in \eqref{eq:hilbert} has a symmetric
unimodal coefficient sequence with center $N/2$, and its coefficient
$\binom{N}{2s}b_s(p)$ is nonnegative.
\end{proof}

\section{Multiplication by a linear form}
\label{sec:lefschetz}

We first reduce the calculation to the diagonal sum $\ell$ and then
describe its action explicitly. A symmetric matrix $H=(h_{ij})$ defines
the directional differential operator
\[
 D_H=\sum_{i\leq j}h_{ij}d_{ij}.
\]
Thus $D_Hf(X)$ is the coefficient of $t$ in $f(X+tH)$. This convention
specifies the off-diagonal coefficients without a trace-pairing
normalization.

\begin{lemma}\label{lem:openorbit}
Over an algebraic closure of $\kk$, all $D_H$ with $H$ nonsingular have
the same ranks of multiplication by every power in $A_N$. These ranks
are the generic ranks and are realized by $\ell=D_I$. If any linear
form over any extension of $\kk$ is weak or strong Lefschetz, then
$\ell$ has the same property.
\end{lemma}

\begin{proof}
Let $T_gf(X)=f(gXg^{\mathsf T})$. Since $T_gD_N=\det(g)^2D_N$,
conjugation $f(\partial)\mapsto T_g^{-1}f(\partial)T_g$ preserves the
apolar ideal. The chain rule gives
$T_g^{-1}D_HT_g=D_{gHg^{\mathsf T}}$.
Over an algebraically closed field of characteristic different from
two, the nonsingular symmetric matrices form a single congruence orbit,
containing $I$.

For each degree and each power, the maximal-rank locus is open, since
the relevant matrix entries are polynomial functions of the coefficients
of the linear form. The nonsingular orbit is open and dense, so its
constant rank is the generic rank. There are only finitely many maps
needed for either Lefschetz property. If their common maximal-rank
locus is nonempty, it meets that orbit. Field extension preserves ranks,
and $\ell$ is defined over $\kk$.
\end{proof}

For a set $J$, let
\[
 B(J)=\kk[y_i:i\in J]/(y_i^2:i\in J),
 \qquad L_J=\sum_{i\in J}y_i.
\]
When $|J|=M$, we also write $B_M$ for this algebra. Grading shifts are
written so that a vector in degree $j$ of $B(J)(-s)$ has total degree
$s+j$.

\begin{proposition}[Boolean decomposition]\label{prop:boolean}
There is an isomorphism of graded $\kk[\ell]$-modules
\begin{equation}\label{eq:boolean}
 A_N\simeq
 \bigoplus_{\substack{E\subseteq[N]\\|E|=2s}}
       U(E)\otimes_\kk B([N]\setminus E)(-s),
\end{equation}
where $U(E)$ is placed in degree zero before the shift, and $\ell$
acts on the summand indexed by $E$ as $1\otimes L_{[N]\setminus E}$.
\end{proposition}

\begin{proof}
Group \eqref{eq:apolarpieces} by its endpoint set $E$. The remaining
vertex degrees are $0$ or $2$, indexed by subsets $T$ of the complement
of $E$. Multiplication by $u_i$ adds $i$ to $T$ if it is not already
present, and otherwise gives zero. It also gives zero for $i\in E$,
because the vertex degree would become three. This is exactly the
asserted action of the square-zero variables $y_i$. Each bracket
matching on $E$ has ordinary degree $s$, giving the shift.
\end{proof}

The decomposition is a statement about a graded module and its linear
operator. It is not an algebra decomposition. All of its summands have
their Hilbert-function center at the same total degree $N/2$.

Define $r_M(j,q)$ to be the rank of multiplication by
$(y_1+\cdots+y_M)^q$ from $(B_M)_j$ to $(B_M)_{j+q}$, and set it equal
to zero for degrees outside $0\leq j\leq j+q\leq M$.

\begin{proposition}[Ranks of all powers]\label{prop:powers}
Suppose $\operatorname{char}\kk=p>2$. For valid degrees,
$r_M(j,q)=0$ if $q\geq p$. If $0\leq q<p$, set
$u=\min(j,M-j-q)$. Then
\begin{equation}\label{eq:inclusionrank}
 r_M(j,q)=
 \sum_{\substack{0\leq i\leq u\\
              p\nmid\binom{u+q-i}{u-i}}}
       \left(\binom{M}{i}-\binom{M}{i-1}\right).
\end{equation}
For every $k,q\geq0$, the generic rank in $A_N$ is
\begin{equation}\label{eq:allranks}
 \rank\bigl(\ell^q:(A_N)_k\to(A_N)_{k+q}\bigr)
 =\sum_{s=0}^{\lfloor N/2\rfloor}
      \binom{N}{2s}b_s(p)r_{N-2s}(k-s,q).
\end{equation}
\end{proposition}

\begin{proof}
In the squarefree monomial bases of $B_M$, multiplication by
$(y_1+\cdots+y_M)^q$ is $q!$ times the inclusion matrix from $j$-subsets
to $(j+q)$-subsets. If $q\geq p$, the power itself is zero, since
$(y_1+\cdots+y_M)^p=\sum_i y_i^p=0$.

For $q<p$, the scalar $q!$ is a unit. Complementation of subsets
transposes the inclusion matrix when needed and replaces $j$ by
$u=\min(j,M-j-q)$. Now $M\geq 2u+q$, so Wilson's rank formula
\cite[Theorem~4.1]{Frankl}, with $a=u+q$ and $b=u$, applies and gives
\eqref{eq:inclusionrank}. Formula~\eqref{eq:allranks} follows from
Proposition~\ref{prop:boolean}, Proposition~\ref{prop:rank}, and the
$\binom{N}{2s}$ choices of $E$. Lemma~\ref{lem:openorbit} identifies
these with the generic ranks.
\end{proof}

\begin{proof}[Proof of Theorem~\ref{thm:lefschetz}]
Suppose first that $N\leq2p-2$. For any Boolean summand with $M\leq N$
and $j<M/2$, one has $j+1<p$. In the $q=1$ case of
\eqref{eq:inclusionrank}, $u=j$ and the divisibility factors are
$j+1-i$ for $0\leq i\leq j$. All are nonzero modulo $p$, so the sum
telescopes to $\binom{M}{j}$. Multiplication is injective on the
increasing half of the Boolean algebra. The perfect pairing between
complementary subsets gives surjectivity on the decreasing half.
Because all shifted Boolean summands in \eqref{eq:boolean} have center
$N/2$, their direct sum has maximal rank in every degree. Thus $\ell$
is weak Lefschetz.

If $N\geq2p-1$, consider the summand with $E=\varnothing$, namely
$B_N$. In it the element
\[
 \ell^{p-1}=(p-1)!\sum_{|T|=p-1}u_T
\]
is nonzero, but is killed by $\ell$. Its degree satisfies
$p-1<N/2$, and Corollary~\ref{cor:unimodal} gives
$\dim(A_N)_{p-1}\leq\dim(A_N)_p$. Maximal rank at this degree would
require injectivity, so $\ell$ fails the weak Lefschetz property.
Lemma~\ref{lem:openorbit} then rules out every other linear form.
At $N=2p-1$, formula~\eqref{eq:inclusionrank} shows that the
endpoint-free summand loses exactly the $i=0$ contribution, of size
one. Every other summand has $M\leq2p-3$ and its middle map has full
rank. The asserted kernel dimension follows.

For strong Lefschetz, if $N<p$, every Boolean summand has $M<p$.
For valid degrees in \eqref{eq:inclusionrank}, the top indices of the
binomial coefficients are less than $p$, so none is divisible by $p$;
also $q!$ is a unit. All powers have maximal rank. More explicitly,
the rank is $\binom{M}{u}$, the smaller of the two dimensions.
For each fixed pair of total degrees $k,k+q$, the side with smaller
dimension is the same in every shifted Boolean summand, since their
centers coincide. Summing therefore preserves maximal rank.

If $N\geq p$, each $d_{ij}^{\,p}$ is zero in $A_N$: every variable in
$D_N$ has exponent at most two, strictly less than $p$. Frobenius gives
$L^p=0$ for every $L\in(A_N)_1$. But $(A_N)_N$ is nonzero, since
$d_{11}\cdots d_{NN}D_N=1$. Hence $L^N$ cannot induce the required
isomorphism $(A_N)_0\to(A_N)_N$. Strong Lefschetz fails.

In characteristic zero, the subset-inclusion maps all have maximal
rank. This follows, for example, from the same inclusion-matrix rank
theorem: for each fixed matrix take a prime larger than $M$, for which
the preceding binomial factors and factorials are units; full rank
modulo that prime implies full rank over $\mathbb Q$. The module
decomposition then proves that $\ell$ is strong Lefschetz.
\end{proof}

\paragraph{Acknowledgment.}
The author acknowledges OpenAI GPT-6 for assistance with research
exploration, proof development, verification code, and manuscript
preparation.

\end{document}